\documentclass{amsart}
\usepackage[utf8]{inputenc}
\usepackage[english]{babel}
\usepackage{amsmath}
\usepackage{amsfonts}
\usepackage{amssymb}
\usepackage{amsthm}
\usepackage{hyperref}
\usepackage{mathtools}
\usepackage{longtable}
\usepackage{graphicx}
\usepackage{bbm}
\usepackage[backrefs,lite]
{amsrefs}

\newcommand{\C}{\mathbb{C}}
\newcommand{\Q}{\mathbb{Q}}
\newcommand{\Z}{\mathbb{Z}}
\newcommand{\T}{\mathbb{T}}

\DeclareMathOperator{\Cl}{Cl}

\DeclareMathOperator{\GL}{GL}

\DeclareMathOperator{\im}{im}
\DeclareMathOperator{\cone}{cone}
\DeclareMathOperator{\Stab}{Stab}
\DeclareMathOperator{\ord}{ord}

\DeclareMathOperator{\age}{age}

\newcommand{\fpart}[1]{\left\{ #1 \right\}} % fractional part in [0,1)
\theoremstyle{plain}
\newtheorem{theorem}{Theorem}[section]
\newtheorem{lemma}[theorem]{Lemma}
\newtheorem{proposition}[theorem]{Proposition}
\newtheorem{corollary}[theorem]{Corollary}

\theoremstyle{definition}
\newtheorem{definition}[theorem]{Definition}

\newtheorem{remark}[theorem]{Remark}
\newtheorem{example}[theorem]{Example}

\newtheorem{algorithm}[theorem]{Algorithm}
\newtheorem{construction}[theorem]{Construction}
\newtheorem{setting}[theorem]{Setting}
\newtheorem{reminder}[theorem]{Reminder}

\title[A canonicity criterion for toric varieties]{A canonicity criterion for toric varieties \\ and the classification of canonical 4-simplices}
\author{Marco Ghirlanda}
\email{marco.ghirlanda@uni-tuebingen.de}
\address{Mathematisches Institut, Auf d. Morgenstelle 10, 72076 T\"ubingen}
\date{}

\begin{document}

\begin{abstract}
Based on the Reid--Shepherd-Barron--Tai criterion for canonical and terminal quotient singularities, we characterize canonicity and terminality of a toric variety in terms of its local class group actions. Specializing it to the Picard number one setting, we arrive at a classification algorithm for canonical and terminal fake weighted projective spaces in any dimension.
In dimension four it gives, up to isomorphism, 710450 canonical fake weighted projective spaces.
We take a look at the corresponding Calabi--Yau hypersurfaces, compute the Fine interior of the associated canonical simplices, and discuss the results.
\end{abstract}

\maketitle

\section{Introduction}

A \emph{canonical polytope} is a lattice polytope with the origin as the only interior lattice point. These polytopes play a
central role in the context
of mirror symmetry as candidate Newton polytopes for Calabi--Yau hypersurfaces~\cites{Bat94,Bat17}.
There are extensive classifications, see \cites{Bat94,KS3,KS4,Ghi} for the reflexive case and \cites{Kas2,Kas3} for the more general canonical case.

In this article, we develop a new approach.
The idea is to use the Reid--Shepherd-Barron--Tai criterion for canonicity
of quotient singularities~\cite{MS84}.
Given an $n$-dimensional  $\Q$-factorial projective toric variety $X$,
we look at Cox's quotient construction $X = \hat X/H$, where
$\hat X \subseteq \C^r$ is an open toric
subvariety and $H$ is a quasitorus acting diagonally on $\C^r$.
For the toric fixed points $x_i \in X$, $i = 1, \ldots, s$, let $H_i \subseteq H$ be the (finite) subgroup having the local class group $\Cl(X,x_i)$ as its character group. 
The criterion is formulated in terms of the \emph{age}
of the $h\in H_i$, i.e. the sum of exponents of the eigenvalues of $h$ acting on $\C^r$; see Theorem~\ref{thm:main}, treating moreover the terminal case.

\begin{theorem}
Let $X$ be a $\Q$-factorial projective toric variety. Then $X$ is canonical
if and only if $\mathrm{age}(h) \ge 1$ for all $1 \neq h\in H_i$, $i=1,\dots,s$.
\end{theorem}

We apply this result to the $\Q$-factorial projective toric varieties
of Picard number one, called \emph{fake weighted projective spaces}.
The adapted version of our criterion, see Theorem~\ref{thm:fwps}, is the base of a
general classification algorithm for canonical and terminal Fano simplices, see Algorithm~\ref{algo:classification}.
The algorithm reproduces Kasprzyk's $225$ canonical fake weighted projective spaces of dimension three
and his $35947$ terminal fake weighted projective spaces of dimension four.
The new achievement is the following.
\begin{theorem}\label{classification}
Up to isomorphism, there are $710450$ canonical fake weighted projective
spaces of dimension four.
\end{theorem}

On a midrange computer, with 16 threads, the algorithm terminates in approximately 12 minutes. In terms of polytopes, the result says that, up to unimodular
equivalence, there are $710450$ canonical simplices of dimension four.
The defining data of these polytopes are made available in~\cite{Zenodo}.

Finally, we consider the toric hypersurfaces with a canonical Newton
polytope.
Batyrev~\cite{Bat23} showed that the Kodaira dimension of their canonical
model is determined by a rational subpolytope of the Newton polytope,
called \emph{Fine interior}.
We computed the Fine interiors of all $710450$ canonical simplices of
dimension four. Their distribution by dimension is
\[
\begin{array}{c|ccccc}
\dim F(\Delta) & 0 & 1 & 2 & 3 & 4\\
\hline
\#\text{ simplices} & 387310 & 112672 & 95713 & 70130 & 44625
\end{array}
\]
In \cite{Bat17}, Batyrev showed that the canonical model associated to a
polytope is Calabi--Yau if and only if the Fine interior of the polytope
has dimension zero, and provided a combinatorial formula for its stringy
Euler number.
We computed the stringy Euler number of our $387310$ simplices with
zero-dimensional Fine interior, and obtained $94233$ distinct values,
of which only $852$ are integral. Their frequency distribution is the following.

\begin{center}\label{fig:euler}
\includegraphics[width=0.7\textwidth]{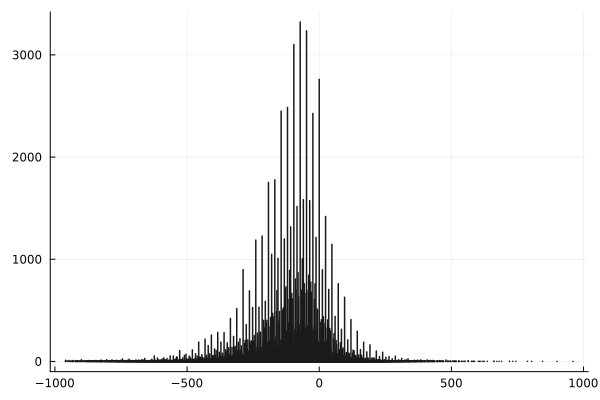}

Figure 1.
\end{center}
\medskip

The paper is organized as follows. In Section~\ref{sec:general} we establish
the canonical and terminal criterion for simplicial toric varieties in terms
of their local class group actions.
In Section~\ref{sec:fwps} we specialize to fake weighted projective spaces
and develop our classification algorithm. In Section~\ref{sec:fine} we compute and analyze
the Fine interiors of the resulting four-dimensional canonical simplices.

The author is grateful to Martin Bohnert and Professor Victor Batyrev for their suggestions concerning Section \ref{sec:fine}. He would also like to thank Professor Jürgen Hausen for his constant help and support throughout the writing of this paper, and for many suggestions that helped make the arguments clearer and more elegant.

\tableofcontents

\section{Characterizing canonicity and terminality}\label{sec:general}

The main result of this section is a criterion for canonicity and terminality of simplicial toric varieties in terms of their local class groups.

We assume the reader is familiar with the basic concepts
of toric geometry~\cite{CoxLittleSchenck}.
Let us fix the setting and the notation for the section.
By $\Sigma$ we denote a simplicial fan in $\Z^n$ and
by $v_1,\dots,v_r$ the primitive generators of its rays.
The \emph{generator matrix} of $\Sigma$ is the $n\times r$ matrix
$$
P=\begin{bmatrix}
    v_1 & \dots & v_r
\end{bmatrix}.
$$
We assume that $P$ is of rank $n$. An associated \emph{degree matrix}
$Q$ is obtained as follows. Consider $K=\Z^r/\im(P^T)$ and write
$K=\Z^k\times\Gamma$ with 
$$
\Gamma=\Z/\mu_1\Z\times\dots\times\Z/\mu_s\Z,
\quad
\mu_s\geq 2,
\quad
\mu_s\mid\mu_{s-1},\dots,\mu_2\mid\mu_1.
$$
Then the $i$-th column of $Q$ is the image $\omega_i \in K$ of $e_i\in\Z^r$
under the canonical projection $\Z^r \to K$:
$$
Q
=
\begin{bmatrix}
    \omega_1 & \dots & \omega_r
\end{bmatrix}
=
\begin{bmatrix}
 w_1 & \dots & w_r \\ \eta_1 & \dots & \eta_r
\end{bmatrix},
\quad
w_i\in\Z^k, \ \eta_i\in\Gamma.
$$
Note that the degree matrix $Q$ is \emph{almost free}, that means
that any $r-1$ of its columns generate $K$.

Let $X$ be the toric variety associated with $\Sigma$.
Then $\Cl(X)\cong K$.
We recall Cox's quotient construction \cite{Cox95}.
Denote by $C(\mu)$ the group of $\mu$-th roots of unity
and consider the characteristic quasitorus, having $K$
as its character group:
$$
H=(\C^*)^k\times C(\mu_1)\times\dots\times C(\mu_s).
$$
Taking the columns $\omega_i = (w_i, \eta_i)$ of the degree
matrix $Q$ as weights, we obtain a diagonal action of $H$
on $\C^r$; explicitly, the elements $h=(t,\zeta) \in H$
act via
$$
(t,\zeta)\cdot z=(t^{w_1}\zeta^{\eta_1}z_1,\dots,t^{w_r}\zeta^{\eta_r}z_r),
\quad
t^{w_i}=t_1^{w_{1i}}\cdots t_k^{w_{ki}},
\quad
\zeta^{\eta_i}=\zeta_1^{\eta_{1i}}\cdots \zeta_s^{\eta_{si}}.
$$
Cox's quotient construction presents $X$ as the quotient of the
$H$-action on the open toric subvariety of $\hat{X} \subseteq \C^r$
given by the fan
$$
\hat{\Sigma}=\{\hat{\sigma}; \ \sigma\in\Sigma\},
\qquad
\hat{\sigma}=\cone(e_i; \ v_i \in\sigma),
$$
where the quotient map $p \colon \hat{X}\rightarrow X$ arises from the
homomorphism $P \colon \Z^r\rightarrow\Z^n$.
As we assume $\Sigma$ to be simplicial, the quotient is geometric, i.e.,
the $p$-fibers are exactly the $H$-orbits.
Note that for the affine toric charts we have
$$
p^{-1}(X(\sigma)) = \hat X(\hat\sigma).
$$

\begin{example}\label{ex:1}
Let $X$ be the toric surface arising from the (unique) complete simplicial fan
with generator matrix
$$
P=\begin{bmatrix}
        1 & 1 & -2\\
        0 & 3 & -3
\end{bmatrix}.
$$
Then the divisor class group of $X$ is $\Cl(X)=K=\Z\times\Z/3\Z$, and a choice
for an associated degree matrix is
$$
Q=\begin{bmatrix}
        1 & 1 & 1\\
        \bar{0} & \bar{1} & \bar{2}
\end{bmatrix}.
$$
The \emph{characteristic quasitorus} is $H=\C^*\times C(3)$.
An element $(t,\zeta)\in H$ acts on $\C^3$ as follows
$$
(t,\zeta)\cdot(z_1,z_2,z_3)=(tz_1,t\zeta z_2,t\zeta^2 z_3).
$$
Furthermore, $\hat{X}=\C^3\setminus\{0\}$ and thus $X = (\C^3\setminus\{0\})/H$.
Taking first the quotient by the connected component $\C^*$, we see
$X\cong \mathbb P^2/C(3)$.
\end{example}

\begin{construction}\label{constr:slice}
Let $\Sigma$ be a simplicial fan in~$\Z^n$ with generator
matrix $P$.
Consider the associated toric variety $X$ with Cox's quotient
presentation $p \colon \hat X \to X$.
For any $n$-dimensional $\sigma \in \Sigma$,
the \emph{affine slice} at the distinguished point
$\hat x_{\hat \sigma} \in \hat X$ is
$$
Y_\sigma
\coloneqq
\{ x \in \C^r; \ x_i=1 \text{ for all } i = 1, \ldots, r \text{ with } \  v_i \not\in \sigma \}
\subseteq
\hat X.
$$
Note that $Y_\sigma$ is an $n$-dimensional
affine subspace of $\C^r$.
For the characteristic quasitorus $H$ and the
\emph{stabilizer} $H_\sigma := \Stab_H(Y_\sigma) \subseteq H$
of $Y_\sigma \subseteq \hat X$, we have
$$
H \cdot Y_\sigma = \hat X(\hat\sigma),
\qquad
X(\sigma) = \hat X(\hat\sigma)/H = Y_\sigma/H_\sigma.
$$
Moreover, the stabilizer $H_\sigma$ of $Y_\sigma$
equals the isotropy group $H_{\hat x_{\hat \sigma}}$
and thus its character group is given by
$$
\mathbb{X}(H_\sigma) = \Cl(X,x_\sigma),
$$
where $\Cl(X,x_\sigma)$ is the local class group of the distinguished
point $x_\sigma \in X$, that means the group of Weil divisors
modulo those being principal near $x_\sigma$. In particular,
as $X$ is $\Q$-factorial, $H_\sigma$ is finite.
\end{construction}

\begin{lemma}\label{lem:finite}
Consider a simplicial fan $\Sigma$ in $\Z^n$ with generator
matrix $P$ and associated degree matrix $Q$.
For any $\sigma \in \Sigma$, the $\Z^k$-parts $w_j$ of
the columns $\omega_j$ of~$Q$ with $v_j \not\in \sigma$
generate $\Q^k$ as a vector space.
\end{lemma}

\begin{proof}
Note that $v_1,\ldots,v_r$ in $\Q^n$ and $w_1,\ldots,w_r$ in $\Q^k$
are Gale dual vector configurations.
In particular, linearly independent families of $v_i$
correspond to generating systems of $w_j$.
\end{proof}

\begin{proof}[Proof of Construction~\ref{constr:slice}]
We show $H \cdot Y_\sigma = \hat X(\hat\sigma)$.
Let $1 \le i_1 < \ldots < i_k \le r$ be the
indices with $v_{i_j} \not \in \sigma$.
Then $\hat X(\hat\sigma)$ consists precisely of
the points $x \in \C^r$ with $x_{i_1}, \ldots, x_{i_k}$
non-zero.
Thus $H \cdot Y_\sigma \subseteq \hat X(\hat\sigma)$.
For the reverse inclusion, let $x \in \hat X(\hat\sigma)$
be given.
Consider the $k \times k$ matrix
$W_\sigma = [w_{i_1}, \ldots, w_{i_k}]^T$.
Lemma~\ref{lem:finite} ensures that~$W_\sigma$ is of 
rank~$k$.
Thus, the homomorphism
$\varphi \colon \T^k \to \T^k$ associated
with $W_\sigma$ is surjective
and we find $t \in \T^k $ with
$\varphi(t) = (x_{i_1}^{-1}, \ldots, x_{i_k}^{-1})$.
For $h = (t,1) \in H$, this means
$h \cdot x \in Y_\sigma$, showing $x \in H \cdot Y_\sigma$.

The fact that $X(\sigma)$ equals $Y_\sigma / H_\sigma$
is then obvious.
Moreover, $H_\sigma = H_{\hat x_{\hat \sigma}}$ is
clear, because $\hat x_{\hat \sigma} \in \C^r$ is the
point with coordinates at $i_1, \ldots, i_k$ equal
to one and all others zero.
Finally, we obtain $\mathbb{X}(H_\sigma) = \Cl(X,x_\sigma)$
by applying \cite{ADHL}*{Proposition~1.6.5.2} to the special
case of our toric variety $X$.
\end{proof}

\begin{construction}\label{constr:ev1}
Consider a simplicial fan $\Sigma$ in $\Z^n$ with generator
matrix $P$ and associated degree matrix $Q$.
Let $\sigma \in \Sigma$ be an $n$-dimensional cone
and let $1 \le i_1 < \ldots < i_k \le r$ be the
indices with $v_{i_j} \not \in \sigma$. These give
us matrices:
$$
W_\sigma \ := \ \begin{bmatrix} w_{i_1} \ldots w_{i_k} \end{bmatrix}^T,
\qquad
E_\sigma \ := \ \begin{bmatrix} \eta_{i_1} \ldots \eta_{i_k} \end{bmatrix}^T.
$$
where $\omega_i = (w_i,\eta_i) \in K = \Z^k \times \Gamma$ is the $i$-th column
of $Q$. For any $\eta \in \Gamma$, we write $\eta = (e_1, \ldots, e_s)$
with $0 \le e_l <  \mu_l$ and set $b(\eta) = (e_1/\mu_1,\ldots,e_s/\mu_s)$.
Then we obtain a monomorphism
$$
\begin{array}{rcl}
\Phi \colon \Z^k/\mathrm{im}(W_\sigma) \times \Gamma & \to & \Q^k/\Z^k \times \Q^s/\Z^s,
\\[5pt]                                                     
(c,\eta) & \mapsto & \left(W_\sigma^{-1} \cdot (c-E_\sigma \cdot b(\eta)), \, b(\eta) \right).
\end{array}
$$
\end{construction}

\begin{proof}
By Lemma~\ref{lem:finite}, the matrix $W_\sigma$ has rank $k$,
hence it is invertible over~$\Q$.
Furthermore, replacing $c$ with $c+W_\sigma u$, where $u\in\Z^k$,
does not change the class of the first component in $\Q^k/\Z^k$.
Hence $\Phi$ is well defined. Clearly $\Phi$ is a homomorphism.
Finally, suppose $\Phi(c,\eta)=0$. Then $\eta=0$, and,
choosing a representative $c' \in \Z^k$ of~$c$, we infer
$c=0$ from
$$
W_\sigma^{-1} \cdot c'  = W_\sigma^{-1}\cdot \left(c'-E_\sigma \cdot b(\eta)\right) = 0 \in \Q^k/\Z^k.
$$
\end{proof}

\begin{proposition}\label{prop:eigenvals}
Consider the setting of Construction~\ref{constr:ev1}.   
With the monomorphism $\Phi$ we obtain an isomorphism of groups
$$
\begin{array}{rcl}
\mathrm{im}(\Phi) & \to & H_\sigma,
\\[5pt]                                                     
(a,b) & \mapsto & h(a,b) := (\exp(2 \pi i a_1), \ldots, \exp(2 \pi i a_k), \, \exp(2 \pi i b_1), \ldots, \exp(2 \pi i b_s)).
\end{array}
$$
Denote by $\alpha_j(a,b)$ the fractional parts of $\omega_{m_j} \cdot (a,b)$,
where $1 \le m_1 < \ldots < m_n \le r$ are the indices with $v_{m_j} \in \sigma$.
Then the eigenvalues of $h(a,b) \in H_\sigma$ acting on~$Y_\sigma \cong \C^n$ are
$$
\exp(2 \pi i \alpha_1(a,b)), \ldots, \exp(2 \pi i \alpha_n(a,b)).
$$
\end{proposition}

\begin{proof}
We begin with a short preparation.
Write the transposed degree matrix as $Q^T= [W,E]$, where $W$
(resp. $E$) has the upper $k$ (resp. lower $s$) rows of $Q$ as
their columns. Denote by $\Q_\mu^s=\Q_{\mu_1}\times\dots\times\Q_{\mu_s}\subset\Q^s$ the subgroup of all tuples $(b_1/\mu_1,\dots,b_s/\mu_s)$ with $b_i\in\Z$. Consider now the homomorphism 
$$
\Psi \colon
\Q^k/\Z^k \times \Q_\mu^s/\Z^s \ \to \ \Q^r/\Z^r,
\qquad
(a,b) \ \mapsto \ Q^T \cdot (a, b) =  W \cdot a + E \cdot b.
$$  
Then $\Psi$ tells us how the elements $h = (t,\zeta)$ of
finite order in $H$ act on $\C^r$.
Any such element can be written uniquely as 
$$
h
\ = \
(t,\zeta)
\ = \
\left(
\exp(2 \pi i a_1), \ldots, \exp(2 \pi i a_k),
\exp(2 \pi i b_1), \ldots, \exp(2 \pi i b_s)
\right),
$$
where $(a_1,\ldots, a_k) \in \Q^k/\Z^k$ and $(b_1,\ldots, b_s) \in \Q_\mu^s/\Z^s$.
With the rows $(w_l,\eta_l)$ of the transpose $Q^T$ and the class of
$\mathfrak{a}_l := (w_l,\eta_l) \cdot (a,b)$ in $\Q/\Z$ we have
$$
h \cdot z
\ = \
(t^{w_1}\zeta^{\eta_1}z_1, \ldots, t^{w_r}\zeta^{\eta_r}z_r)
\ = \
\left(\exp(2 \pi i\mathfrak{a}_1)z_1, \ldots, \exp(2 \pi i \mathfrak{a}_r)z_r \right).
$$

Now, by Construction~\ref{constr:slice}, the group $H_\sigma \subseteq H$
consists precisely of the finite order elements of $H$ having eigenvalue
$1$ at the coordinates $l$ with $v_l \not \in \sigma$.
Using the definitions of $\Phi$ and $\Psi$, we see that any $h(a,b)$ with $(a,b)\in\im(\Phi)$
is actually an element of~$H_\sigma$.
Obviously, $\Phi$ is injective.
For surjectivity, let $h = (t,\zeta) \in H_\sigma$ be given.
Write $t_l=\exp(2\pi i a_l)$ and  $\zeta_l=\exp(2\pi i b_l)$
with $a_l \in \Q/Z$ and $b_l \in \Q_{\mu_l}/\Z$.
We claim
$$
c := W_\sigma \cdot a + E_\sigma \cdot b \in \Z^k.
$$
Indeed, as $h$ has eigenvalue 1 at the coordinates $l$
with $v_l \not \in \sigma$, the corresponding exponents
$\mathfrak{a}_l = (w_l,\eta_l) \cdot (a,b)$ must be integral.
Thus,  $c_l=\mathfrak{a}_l$ gives the claim.
Finally, using again the above presentation of $h \cdot z$ for
$z \in \C^r$, we see that $H_\sigma$ acts on the affine slice
$Y_\sigma$ as claimed in the assertion, with $\alpha_j=\mathfrak{a}_{m_j}$.
\end{proof}

\begin{example}\label{ex:2}
Consider again the toric surface $X$ from Example~\ref{ex:1}.
Recall that we have $\Cl(X) = K = \Z\times\Z/3\Z$ and
generator and degree matrix are given by
$$
P=\begin{bmatrix}
1 & 1 & -2
\\
0 & 3 & -3
\end{bmatrix},
\qquad\qquad
Q
=
\begin{bmatrix}
1 & 1 & 1
\\
\bar{0} & \bar{1} & \bar{2}
\end{bmatrix}.
$$
We go through Construction~\ref{constr:ev1} with
the maximal cone $\sigma=\cone(v_1,v_2)$.
The corresponding matrices arise from the third
column of $Q$:
\[
W_\sigma=\begin{bmatrix}1\end{bmatrix},
\qquad
E_\sigma=\begin{bmatrix}\bar{2}\end{bmatrix}.
\]
In particular, the factor group $\Z/\im(W_\sigma)$ is trivial
and $\Cl(X,x_\sigma)=\Gamma=\Z/3\Z$.
The homomorphism from~\ref{constr:ev1} is explicitly given
as 
\[
\Phi\colon \Gamma \to \Q/\Z\times \Q_3/\Z,
\qquad
\eta \mapsto (-2b(\eta),\,b(\eta)).
\]
We determine the eigenvalues of $\eta=\bar{1} \in\Gamma=\Z/3\Z$.
Since $b(\eta)=1/3$, we obtain
\[
(a,b)
:=
\Phi(0,\bar{1})
=
\left(-\tfrac{2}{3},\tfrac{1}{3}\right)
=
\left(\tfrac{1}{3},\tfrac{1}{3}\right)
\in
\Q/\Z\times \Q_3/\Z.
\]
Now we determine the weights of the action of $H_\sigma$ on
$\C^3$.
The fractional parts $\alpha_l$ of the scalar products
$(w_l,\eta_l) \cdot (a,b)$ for $l=1,2,3$ are
\[
\alpha_1 = (1,\bar 0) \cdot (a,b) = 0+\tfrac{1}{3} = \tfrac{1}{3} \in \Q/\Z,
\]
\[
\alpha_2 = (1,\bar 1) \cdot(a,b) = \tfrac{1}{3}+\tfrac{1}{3} = \tfrac{2}{3} \in \Q/\Z,
\]
\[
\alpha_3 = (1,\bar 2) \cdot(a,b) = \tfrac{1}{3}+\tfrac{2}{3} = 0 \in \Q/\Z.
\]
Thus, according to Proposition~\ref{prop:eigenvals}, the generator $\zeta = \exp(2\pi i/3)$
of $H_\sigma = C(3)$ acts on $\C^3$ via
\[
\zeta \cdot z = (\zeta z_1, \zeta^2 z_2, z_3).
\]
The affine slice at $\hat x_{\hat \sigma} = (0,0,1)$ is the
plane $Y_\sigma = (0,0,1) + V(z_3)$
and~$\zeta$ (as well as~$\zeta^2$) has the eigenvalues
\[
\exp(2\pi i/3),\qquad \exp(4\pi i/3).
\]
The quotient $\hat X({\hat \sigma}) / H = Y_\sigma/ H_\sigma$ has
an $A_2$-singularity and, in particular, is Gorenstein.
The same holds for the remaining two maximal cones.
\end{example}

Recall that the \emph{age} of an element $g \in \GL(n,\C)$ with
eigenvalues $\exp(2\pi i\alpha_j)$, where $j = 1, \ldots, n$
and $0\le \alpha_j<1$, is the sum $\age(g)\coloneqq \alpha_1+\dots+\alpha_n$.
Here comes the main result of this section:

\begin{theorem}\label{thm:main}
Let $X$ be the toric variety arising from a simplicial fan~$\Sigma$
in $\Z^n$ and let $\sigma \in \Sigma$ be an $n$-dimensional cone.
Consider the affine slice $Y_\sigma$ and the action of its stabilizer
$H_\sigma$ from Construction~\ref{constr:slice}. Then
\begin{enumerate}
\item
$x_\sigma \in X$ is canonical if and only if $\age(h) \ge 1$ for all $1 \ne h \in H_\sigma$,
\item
$x_\sigma \in X$ is terminal if and only if $\age(h) > 1$ for all $1 \ne h \in H_\sigma$.
\end{enumerate}
\end{theorem}

\begin{proof}
Construction~\ref{constr:slice} tells us that the affine toric
chart $X(\sigma)$ hosting the distinguished point $x_\sigma$
is isomorphic to the quotient $Y_\sigma/H_\sigma$, with the slice
$Y_\sigma \cong \C^n$ and the finite abelian group $H_\sigma$.
By Proposition~\ref{prop:eigenvals}, the eigenvalues of $h \in H_\sigma$
acting on $Y_\sigma$ are given as
\[
\exp(2\pi i\alpha_1(a,b)),\dots,\exp(2\pi i\alpha_n(a,b)),
\]
where the $\alpha_l(a,b)$ can be extracted from the degree matrix
and a suitable presentation of $h \in H_\sigma$.
In particular we have
$$
\age(h)=\alpha_1(a,b)+\dots+\alpha_n(a,b).
$$

The idea is to apply the Reid--Shepherd-Barron--Tai criterion~\cite{MS84},
which directly yields the assertion, provided that no
$1\neq h \in H_\sigma$ is a quasi-reflection on $Y_\sigma$.
Let us verify the latter.
Any $h \in H_\sigma \subseteq H$ acts diagonally on the whole
$\C^r$ and its fixed point set is a coordinate subspace $V^h \subseteq \C^r$.
Moreover, from $h \in H_\sigma$ and the definition of $H_\sigma$
we infer
$$
\hat x_{\hat\sigma} \in V(z_l; \ v_l \in \sigma) \subseteq V^h.
$$
Now assume that $h$ is a quasi-reflection on $Y_\sigma \cong \C^n$.
Then the fixed point set $Y_\sigma^h \subseteq Y_\sigma$
is an affine subspace of dimension $n-1$ and by the
definition of $Y_\sigma$, we have
$$
Y_\sigma^h \subseteq Y_\sigma = \hat x_{\hat\sigma} + V(z_l; \ v_l \not\in \sigma).
$$
Thus, $Y_\sigma^h - \hat x_{\hat\sigma}$ is a vector subspace
of dimension $n-1$ in $V(z_l; \ v_l \not\in \sigma)$ consisting
of $h$-fixed points.
Consequently, the fixed point set $V^h \subseteq \C^r$ of $h$ is
a coordinate subspace of dimension $n-1 + k = r-1$.
This contradicts the almost freeness of~$Q$, which ensures
that $H$ acts freely outside the union of all
coordinate subspaces of dimension $r-2$.
\end{proof}

\begin{corollary}
Let $X$ be a $\Q$-factorial complete toric variety arising from a
fan~$\Sigma$.
\begin{enumerate}
\item
$X$ is canonical if and only if for all maximal cones
$\sigma\in\Sigma$ every non-trivial $h \in H_\sigma$
satisfies $\age(h)\ge 1$.
\item
$X$ is terminal if and only if for all maximal cones
$\sigma\in\Sigma$ every non-trivial $h \in H_\sigma$
satisfies $\age(h) > 1$.
\end{enumerate}
\end{corollary}

\begin{proof}
This follows from Theorem~\ref{thm:main} and the facts
that $X$ is covered by the affine toric charts $X(\sigma)$
stemming from the maximal cones $\sigma \in \Sigma$
and each such $X(\sigma)$ is canonical (terminal)
if and only if $x_\sigma$ is canonical (terminal).
\end{proof}

\begin{example}\label{ex:3}
Consider once more the toric surface $X$ from Example~\ref{ex:1}.
Due to Example~\ref{ex:2}, all nontrivial elements of $H_\sigma$,
where $\sigma$ runs through the three maximal cones, are of age 1.
Thus $X$ is canonical (as any Gorenstein toric surface), but not terminal.
\end{example}

\section{Classifying fake weighted projective spaces}\label{sec:fwps}

We focus on the case that our toric variety $X$ is $\Q$-factorial, projective and
of Picard number one. These varieties are called \emph{fake weighted projective spaces}
(fwps), and each of them is Fano, i.e., the anticanonical divisor is ample.
The main result of the section is Algorithm~\ref{algo:classification}, which
classifies the canonical fwps in a given dimension.

\begin{setting}\label{setting-sec-2}
For an fwps $X$ of dimension $n$, our notation and shape of the
$n \times (n+1)$ generator matrix and the degree matrix throughout
the section are:
$$
P=\begin{bmatrix}
v_0 & \dots & v_n
\end{bmatrix},
\qquad
Q
=
\begin{bmatrix}
w_0 & \dots & w_n
\\
\eta_0 & \dots & \eta_n
\end{bmatrix},
\quad
w_0 \le \ldots \le w_n \in \Z_>0, \ \eta_i\in \Gamma,
$$
where $\Gamma = \Z/\mu_1\Z\times\dots\times\Z/\mu_s\Z$ with $\mu_s\geq 2$ and
$\mu_s\mid\mu_{s-1},\dots,\mu_2\mid\mu_1$.
The fan $\Sigma$ of $X$ has the maximal cones $\sigma_i=\cone(v_j;\ j\neq i)$,
where $i=0,\dots,n$.
\end{setting}

We specialize Theorem~\ref{thm:main} to fwps and, as it turns out to be
useful for computational purposes, formulate the characterizing conditions
for canonicity and terminality entirely in terms of integers.

\begin{theorem}\label{thm:fwps}
Let $X$ be an $n$-dimensional fwps.
For $0 \le i \le n$, $0 \le c \le w_i-1$ and~$b=(b_1/\mu_1,\dots,b_s/\mu_s)$
with~$0\le b_\ell\le \mu_\ell-1$ such that $(c,b)\neq(0,0)$ let
$R_{ij}(c,b)$ denote the smallest non-negative integer with
$$
R_{ij}(c,b)=\mu_1\bigl(w_jc+(w_i\eta_j-w_j\eta_i)\cdot b\bigr)\mod \mu_1w_i.
$$
Then $X$ is canonical if and only if $\Sigma_{j\neq i}R_{ij}(c,b) \ge \mu_1w_i$
for all $i, c, b$ as above
and~$X$ is terminal if and only if $\Sigma_{j\neq i}R_{ij}(c,b) > \mu_1w_i$
for all $i, c, b$ as above.
\end{theorem}

\begin{proof}
First note that for the maximal cone $\sigma_i=\cone(v_j;\ j\neq i)$,
the associated matrices $W_{\sigma_i}$ and $E_{\sigma_i}$ from
Construction~\ref{constr:ev1} are just the integer $w_i$
and the column vector $\eta_i$.
Moreover, Construction~\ref{constr:ev1} and Proposition~\ref{prop:eigenvals},
uniquely represent the elements of $H_{\sigma_i}$ as $h =h(a,b)$ with
$$
a=\frac{1}{w_i}(c-\eta_i\cdot b)\in\Q/\Z
$$
and $c$, $b$ as in the assertion.
For $j \neq i$, Proposition~\ref{prop:eigenvals} shows that
the corresponding eigenvalue of $h=h(a,b)$ on $Y_{\sigma_i}$ has
fractional part
\[
A_{ij}(a,b)
\coloneqq
\fpart{w_ja+\eta_j\cdot b}
=
\fpart{\tfrac{w_j}{w_i}c+\tfrac{1}{w_i}(w_i\eta_j-w_j\eta_i)\cdot b}.
\]
We conclude $\mathrm{age}(h) = \Sigma_{j\neq i} A_{ij}(a,b)$.
Moreover, $R_{ij}(c,b)=\mu_1w_iA_{ij}(a,b)$ holds.
Together with Theorem~\ref{thm:main}, this yields the
assertion.
\end{proof}

The following observation gives us in particular easy-to-compute
necessary criteria for canonicity and terminality for a degree matrix
by applying Theorem~\ref{thm:fwps} to submatrices obtained
by removing rows.
By a \emph{canonical (terminal) degree matrix} we mean the degree
matrix of a canonical (terminal) fwps.

\begin{proposition}\label{prop:partial-rows}
Let $Q$ be a canonical (terminal) degree matrix and $Q'$ be the
submatrix obtained from $Q$ by deleting some of the lower $s$ rows.
Then $Q'$ is a canonical (terminal) degree matrix.
\end{proposition}

\begin{proof}
Let $Q'$ arise from $Q$ by deleting the rows number $k_1, \dots, k_t$
from the lower~$s$ ones.
Then the characterizing conditions for $Q'$ from Theorem~\ref{thm:fwps} are
precisely the characterizing conditions for $Q$ with
$b_{k_1} = \dots = b_{k_t} = 0$.
\end{proof}

We call the first row $w = (w_0, \ldots,w_n)$ of a degree matrix $Q$
the \emph{weight vector} of $Q$ or of the associated fwps $X$.
By Proposition~\ref{prop:partial-rows}, if $X$ is canonical
(terminal), then the weight vector defines a canonical (terminal)
weighted projective space.

\begin{reminder}\label{rem:bounds}
Let $X$ be a canonical fwps with weight vector $w = (w_0, \ldots,w_n)$ and 
denote by~$s_i$ the $i$-th Sylvester number. Then \cite{AKN}*{Theorem 2.7} and
\cite{Kas1}*{Corollary 2.11} provide us with the following bounds:
$$
w_0 + \ldots + w_n \ \leq \ (s_{n+1}-1)^n,
\qquad
\mu_1\cdots\mu_s \ \leq \ \frac{(w_0 + \ldots + w_n)^{n-1}}{w_1  \cdots w_n}.
$$
The first bound gives a finite search space for the canonical weight
vectors.
One can further restrict by upper bounds on the ratios $w_i/(w_0 + \ldots + w_n)$,
see \cite{Kas1}*{Theorem 3.5}. This way, Kasprzyk classified all canonical
weighted projective spaces in dimension four, see \cite{Kas2}.
\end{reminder}

Let $w=(w_0,\dots,w_n)$ be a weight vector. By a
\emph{$\mu$-torsion vector for $w$} we mean an
$\eta=(\eta_0,\dots,\eta_n) \in (\Z/\mu\Z)^{n+1}$ 
such that the $2 \times (n+1)$ matrix $Q$ with first
row~$w$ and second row $\eta$ is a degree matrix.
In this situation, we call $\eta$
\emph{canonical (terminal)} if $Q$ is so.

\begin{definition}
Let $w = (w_0,\ldots,w_n)$ be a weight vector.
\begin{enumerate}
\item
Two $\mu$-torsion vectors $\eta, \eta'$ for $w$ are \emph{row-equivalent}
if there is an automorphism of $\Z\times\Z/\mu\Z$ sending $(w_i,\eta_i)$
to $(w_i,\eta'_i)$ for $i=0,\dots,n$.
\item
A $\mu$-torsion vector $\eta$ is \emph{row-minimal} if it is minimal w.r.t.
the lexicographical order among all $\mu$-torsion vectors for $w$ being
row-equivalent to $\eta$.
\end{enumerate}
\end{definition}

\begin{lemma}\label{lem:lex_min}
Let $w=(w_0,\dots,w_n)$ be a weight vector and
$\eta=(\eta_0,\dots,\eta_n)$ a $\mu$-torsion vector for $w$.
For $i = 0, \ldots, n$ set
$$
g_{-1}=\mu,\quad g_i=\gcd(\mu,w_0,\dots,w_i),\quad d_i\coloneqq\dfrac{\mu g_i}{g_{i-1}}.
$$
Any $\kappa \in(\Z/\mu\Z)^*$ admits a unique $\mu$-torsion
vector $\eta(\kappa)$ for $w$ row-equivalent to~$\kappa \eta$ with
$0 \le \eta(\kappa)_i < d_i$ for $i=0,\dots,n$.
The following statements are equivalent:
\begin{enumerate}
\item
The $\mu$-torsion vector $\eta$ is row-minimal.
\item
We have $0\leq\eta_i<d_i$ for $i=0,\dots,n$ and $\eta\leq \eta(\kappa)$ for all  $\kappa\in(\Z/\mu\Z)^*$.
\end{enumerate}
\end{lemma}

\begin{proof}
By \cite{Ghi}*{Theorem 2.1}, the $\mu$-torsion vectors row-equivalent to
$\kappa\eta$ are exactly the vectors $\kappa\eta+mw$ with $m\in\Z$.
We fix $\kappa\in(\Z/\mu\Z)^*$ and normalize the coordinates successively via the choice of $m$.
Assume that the first $i-1$ coordinates have already been normalized. Replacing $m$ by
$m+t\,\mu/g_{i-1}$ leaves the first $i-1$ coordinates unchanged and changes the $i$-th coordinate by
multiples of $(\mu/g_{i-1})w_i$. As
\[
\gcd\!\left(\mu,\frac{\mu}{g_{i-1}}w_i\right)=\frac{\mu g_i}{g_{i-1}}=d_i,
\]
there is a choice of $t$ with $0\le \kappa\eta_i+(m+\mu/g_{i-1})w_i<d_i$. By induction on $i$, this gives a unique
normalized representative $\eta(\kappa)$.

Every row-equivalence class contains exactly one such normalized
representative. Therefore $\eta$ is row-minimal if and only if
$\eta$ itself is normalized, i.e. $0\le \eta_i<d_i$ for all $i$,
and is lexicographically minimal among the representatives
$\eta(\kappa)$, that is, $\eta\le \eta(\kappa)$ for all
$\kappa\in(\Z/\mu\Z)^*$.
\end{proof}

We extend the notions of row-minimality and minimality to degree
matrices $Q$ with columns in $\Z \times \Gamma$ as in
Setting~\ref{setting-sec-2}.

\begin{definition}
Consider degree matrices $Q = [\omega_0,\ldots, \omega_n]$
and $Q' = [\omega_0',\ldots, \omega_n']$ with columns in $\Z \times \Gamma$.
\begin{enumerate}
\item
$Q, Q'$ are \emph{row-equivalent} if there is an automorphism
$\varphi$ of $\Z \times \Gamma$ such that $\omega_i' = \varphi(\omega_i)$
for $i = 0, \ldots, n$.
\item
$Q, Q'$ are \emph{equivalent} if they are row-equivalent up to column permutations fixing their (common) weight vector.
\item
$Q$ is \emph{row-minimal} (\emph{minimal}) if, w.r.t. the row-wise
lexicographical order, $Q\leq Q'$ holds for all $Q'$ row-equivalent
(equivalent) to $Q$.
\end{enumerate}
\end{definition}

\begin{remark}
Two degree matrices are equivalent if and only if the associated
fwps are isomorphic to each other.
\end{remark}

\begin{remark}
Theorem~2.1 of \cite{Ghi} provides an explicit set of generators
for the automorphism group of any finitely generated abelian group.
In the case $\Z \times \Gamma$, this allows an efficient encoding
of its automorphism group.
\end{remark}

We are now ready to state our classification algorithm for canonical (terminal) fwps.
The idea is to start with a canonical (terminal) weight vector and to adjoin
torsion rows iteratively, retaining only those extensions that preserve
canonicity (terminality). At each stage, it is enough to test the torsion rows
that survived the previous one.
We denote by $Q_\eta$ the stack matrix obtained from $Q$ by adjoining $\eta$ as
the last row.

\begin{algorithm}\label{algo:classification}
\emph{Input}: a positive integer $n\in\Z_{\geq0}$.

\smallskip

%\noindent\emph{Step 0:}
\noindent\emph{Initialization}
\begin{enumerate}
\item
Compute the set $\mathcal{Q}_0$ of all canonical (terminal) weight
vectors of length~$n+1$ using Reminder~\ref{rem:bounds}.
\item
For each $w\in\mathcal{Q}_0$, compute the set $M(w)$ of all pairs
$(\eta,\mu)$ where $\eta$ is a row-minimal canonical (terminal)
$\mu$-torsion vector for $w$, using Lemma \ref{lem:lex_min}
and Reminder \ref{rem:bounds}.
\end{enumerate}
    
%\noindent\emph{Step $i \ge 1$:}
\noindent\emph{Recursion}
\begin{enumerate}
\item
For all $Q\in\mathcal{Q}_{i-1}$ and all $(\eta,\mu)\in M(Q)$, compute
the set $\mathcal{Q}_i$ of all minimal matrices of the form~$Q_\eta$.
\item For all $Q_\eta\in\mathcal{Q}_i$, compute the set $M(Q_\eta)$ of
all pairs $(\eta',\mu')\in M(Q)$ such that $\mu'\mid\mu$, the bound
of Reminder \ref{rem:bounds} is respected, and the matrix $(Q_\eta)_{\eta'}$
is almost-free, canonical (terminal), and row-minimal.
\item
Terminate the algorithm if $M(Q_\eta)$ is empty for all $Q_\eta\in\mathcal{Q}_i$
or if $i=n-1$.
\end{enumerate}
\noindent\emph{Output:} The set $\bigcup_i\mathcal{Q}_i$ of all minimal
canonical (terminal) degree matrices for $n$-dimensional fwps.
\end{algorithm}

\begin{remark}
To isolate the row-minimal (minimal) matrices in a set $\mathcal{Q}$ we
implement a sieve method.
We iteratively select a matrix in $\mathcal{Q}$ and compute its entire
equivalence class.
We then identify the row-minimal (minimal) representative $Q$, add it
to the target set $\mathcal{Q'}$, and discard from $\mathcal{Q}$ all
matrices row-equivalent (equivalent) to $Q$.
This process is repeated until $\mathcal{Q}$ is empty.
\end{remark}

\begin{proof}[Proof of Algorithm \ref{algo:classification}]
Every matrix output by Algorithm \ref{algo:classification} is minimal
and canonical (terminal).
Conversely, let $Q$ be a minimal canonical (terminal) degree matrix
with $s+1$ rows.
By almost freeness, we have $s \le n-1$. Write $Q_i$ for the submatrix
consisting of the first $i+1$ rows of $Q$. Since $Q$ is minimal, $Q_i$ is also minimal. Hence, by Proposition \ref{prop:partial-rows} and induction on $i$,
we have $Q_i\in\mathcal{Q}_i$ for all $i$. In particular,
$Q = Q_s \in \mathcal{Q}_s$. 
\end{proof}
\section{Fine interiors}\label{sec:fine}

We determine the Fine interiors of all four-dimensional
canonical lattice polytopes.
First recall that the Fine interior of an $n$-dimensional
lattice polytope $\Delta \subseteq \Q^n$, as introduced
in~\cite{Fin}, is the following  
$$
F(\Delta)\coloneqq
\bigcap_{0\neq v\in \Z^n}
\left\{
x \in \Q^n \;\middle|\; \langle x,v\rangle \geq \ord_\Delta(v)+1
\right\},
\qquad
\ord_\Delta(v)\coloneqq \min_{x\in \Delta}\langle x,v\rangle.
$$
The Fine interior is a polytope satisfying
$\Delta^\circ\cap \Z^n \subseteq F(\Delta) \subsetneq \Delta$.
Whereas in dimensions one and two the Fine interior equals
the convex hull over $\Delta^\circ\cap \Z^n$, this is no
longer true in higher dimensions, where $F(\Delta)$ has
rational vertices in general.

The Fine interior plays an important role in the birational
geometry of toric hypersurfaces.
Take any $\Delta \subseteq \Q^n$, let $f$ be a $\Delta$-non-degenerate
polynomial in $n$-variables in the sense of~\cite{Kho},
and consider the  affine hypersurface
$$
Z_\Delta = V(f) \subseteq (\C^*)^n.
$$
Following Batyrev~\cite{Bat17}*{Definition 2.15}, a \emph{canonical model}
of $Z_\Delta$ is a projective normal $\Q$-Gorenstein canonical variety
$\widetilde Z_\Delta$ birational to $Z_\Delta$, such that the
linear system associated with a sufficiently large multiple of
the canonical divisor is base point free.

By~\cite{Bat17}*{Theorem 2.18}, the hypersurface $Z_\Delta$ admits a
canonical model if and only if the Fine interior
$F(\Delta)$ is non-empty.
Now, let $Z_\Delta$ admit a canonical model.
Then \cite{Bat17}*{Theorem~2.23} shows that $\widetilde Z_\Delta$
is birational to a Calabi--Yau variety with at worst
Gorenstein canonical singularities if and only if $F(\Delta)$
consists of a single lattice point.
Moreover, due to \cite{Bat23}*{Theorem~9.2} the Kodaira dimension
of $\widetilde Z_\Delta$ is the minimum of $\dim F(\Delta)$ and $n-1$.

In~\cite{BKS}, Batyrev, Kasprzyk and Schaller computed the Fine
interiors of all canonical Fano polytopes from Kasprzyk's list~\cite{Kas3}.
They found $665599$ canonical polytopes with zero-dimensional Fine
interior, $9040$ with one-dimensional Fine interior, and $49$ with
three-dimensional Fine interior. No canonical polytope of dimension
three has two-dimensional Fine interior.
We did the same with our canonical $4$-simplices:

\begin{proposition}
The Fine interiors of the $710450$ canonical simplices
of dimension four are distributed by dimension as follows:
$$
\begin{array}{c|ccccc}
\dim F(\Delta) & 0 & 1 & 2 & 3 & 4\\
\hline
\#\text{ simplices} & 387310 & 112672 & 95713 & 70130 & 44625
\end{array}
$$
\end{proposition}

Note that, although the zero-dimensional case is still the most frequent,
it is less dominant than in dimension three. Moreover, all dimensions occur.
In the zero-dimensional case, Batyrev gave in \cite{Bat17} a combinatorial
formula for the stringy Euler number $e_{\mathrm{str}}(\widetilde Z_\Delta)$
of the anticanonical model.
Applying this formula to our $4$-simplices, we obtain the following.

\begin{proposition}
For the $387310$ canonical simplices $\Delta$ with zero-dimensional
Fine interior, consider the set of Calabi--Yau models $\widetilde Z_\Delta$.
On this set, the stringy Euler number $e_{\mathrm{str}}(\widetilde Z_\Delta)$
attains $94233$ distinct values, of which only $852$ are integral.
Their frequency distribution is shown in Figure \ref{fig:euler}.
\end{proposition}

The Fine interior determines further birational invariants of the associated
canonical model.
In our setting, there are by far less Fine interiors than we have $4$-simplices.
More precisely, our computations yield the following.

\begin{proposition}\label{prop:fine-int}
The canonical simplices of dimension four give rise, up to unimodular equivalence,
to $46$, $2178$, $7825$, and $16072$ distinct Fine interiors of dimensions $1$,
$2$, $3$, and $4$, respectively.
\end{proposition}

\begin{example}
The $46$ one-dimensional Fine interiors from Proposition~\ref{prop:fine-int}
are
% (up to unimodular equivalence)
the segments $[-p,q] \times 0$ in $\Q^4$, where $[-p,q]$ is one of

\begin{longtable}{cccccc}
$[-7/8, 7/8]$ & $[-6/7, 6/7]$ & $[-6/7, 7/8]$ & $[-5/6, 5/6]$ & $[-5/6, 7/8]$ & $[-4/5, 4/5]$ \\
$[-4/5, 5/6]$ & $[-4/5, 6/7]$ & $[-4/5, 7/8]$ & $[-3/4, 3/4]$ & $[-3/4, 4/5]$ & $[-3/4, 5/6]$ \\
$[-3/4, 6/7]$ & $[-3/4, 7/8]$ & $[-2/3, 2/3]$ & $[-2/3, 3/4]$ & $[-2/3, 4/5]$ & $[-2/3, 5/6]$ \\
$[-2/3, 6/7]$ & $[-2/3, 7/8]$ & $[-1/2, 1/2]$ & $[-1/2, 2/3]$ & $[-1/2, 3/4]$ & $[-1/2, 4/5]$ \\
$[-1/2, 5/6]$ & $[-1/2, 6/7]$ & $[-1/2, 7/8]$ & $[-1/3, 1/3]$ & $[-1/3, 1/2]$ & $[-1/3, 2/3]$ \\
$[-1/3, 3/4]$ & $[-1/3, 5/6]$ & $[-1/3, 6/7]$ & $[-1/3, 7/8]$ & $[-1/4, 1/4]$ & $[-1/4, 1/2]$ \\
$[-1/4, 3/4]$ & $[0, 1/4]$ & $[0, 1/3]$ & $[0, 1/2]$ & $[0, 2/3]$ & $[0, 3/4]$ \\
$[0, 4/5]$ & $[0, 5/6]$ & $[0, 6/7]$ & $[0, 7/8]$ & & \\
\end{longtable}
The associated canonical model $\widetilde Z_\Delta$ has Kodaira dimension one, and,
using~\cite{Gie}*{Theorem 1.1}, we see that the plurigenera are given by 
$$
P_m(\widetilde Z_\Delta)=\lfloor mp\rfloor+\lfloor mq\rfloor+1.
$$
In particular, if two simplices have distinct segments as their Fine interiors,
then their canonical models are not birationally equivalent to each other.
\end{example}
% --------------------------------------------------
% Bibliography
% --------------------------------------------------

\end{document}